\documentclass[a4paper,10pt,leqno,twoside]{tobart}

\usepackage[english]{babel} \usepackage{inputenc, amsmath, amssymb, latexsym,
  epic, epsfig, rotating, fancyhdr, amsthm, pifont, mathabx}

\newcommand{\eqand}{\ensuremath{\quad \textrm{and} \quad}}

\newcommand{\foot}{\footnote}

\newcommand{\ssq}{\ensuremath{\subseteq}}
\newcommand{\smin}{\ensuremath{\setminus}}

\newcommand{\wh}{\ensuremath{\widehat}}

\newcommand{\kreis}{\ensuremath{\mathbb{T}^{1}}}

\newcommand{\torus}{\ensuremath{\mathbb{T}^2}}

\newcommand{\twomatrix}[4]{\ensuremath{\left(\begin{array}{cc} #1 & #2 \\ #3 &
      #4 \end{array}\right)}}

\newcommand{\alphlist}{\begin{list}{(\alph{enumi})}{\usecounter{enumi}\setlength{\parsep}{2pt}
      \setlength{\itemsep}{1pt} \setlength{\topsep}{5pt}
      \setlength{\partopsep}{3pt}}}
\newcommand{\arablist}{\begin{list}{(\arabic{enumi})}{\usecounter{enumi}\setlength{\parsep}{2pt}
          \setlength{\itemsep}{1pt} \setlength{\topsep}{5pt}
          \setlength{\partopsep}{3pt}}}
\newcommand{\romanlist}{\begin{list}{(\roman{enumi})}{\usecounter{enumi}\setlength{\parsep}{2pt}
              \setlength{\itemsep}{1pt} \setlength{\topsep}{5pt}
              \setlength{\partopsep}{3pt}}}
\newcommand{\Romanlist}{\begin{list}{(\Roman{enumi})}{\usecounter{enumi}\setlength{\parsep}{2pt}
              \setlength{\itemsep}{1pt} \setlength{\topsep}{5pt}
              \setlength{\partopsep}{3pt}}}
\newcommand{\bulletlist}{\begin{list}{$\bullet$}{\setlength{\parsep}{2pt}
                \setlength{\itemsep}{1pt} \setlength{\topsep}{5pt}
                \setlength{\partopsep}{3pt}\setlength{\leftmargin}{15pt}}} 
\newcommand{\Alphlist}{\begin{list}{(\Alph{enumi})}{\usecounter{enumi}\setlength{\parsep}{2pt}
      \setlength{\itemsep}{1pt} \setlength{\topsep}{5pt}
      \setlength{\partopsep}{3pt}}}
 \newcommand{\listend}{\end{list}}

\newcommand{\T}{\ensuremath{\mathbb{T}}}

\newcommand{\N}{\ensuremath{\mathbb{N}}} 
\newcommand{\R}{\ensuremath{\mathbb{R}}}
\newcommand{\Z}{\ensuremath{\mathbb{Z}}}
\newcommand{\Q}{\ensuremath{\mathbb{Q}}}

\newcommand{\A}{\ensuremath{\mathbb{A}}}

\newcommand{\cA}{\mathcal{A}}

\newcommand{\cC}{\mathcal{C}}

\newcommand{\cU}{\mathcal{U}}

\newtheoremstyle{tobthm}{3pt}{3pt}{\itshape}{0pt}{\bfseries}{.}{0.5eM}{}
\theoremstyle{tobthm}

\newtheorem{definition}{Definition}[section]
\newtheorem{thm}[definition]{Theorem}

\newtheorem{lem}[definition]{Lemma}
\newtheorem{cor}[definition]{Corollary}  
\newtheorem{prop}[definition]{Proposition}

\newtheoremstyle{tobrem}{3pt}{3pt}{\normalfont}{0pt}{\bfseries}{.}{0.5em}{}
\theoremstyle{tobrem}

\numberwithin{equation}{section}
\numberwithin{figure}{section}

\title{\large\textsc{Irrational rotation factors for conservative
    torus homeomorphisms}}

\author{T.~J\"ager\thanks{Institute of Mathematics, FSU Jena, Germany. Email:
    {\tt Tobias.Oertel-Jaeger@tu-dresden.de}}\ \ and
 F.~Tal\thanks{Universidade de S\~{a}o Paulo, Brasil. Email: {\tt fabiotal@ime.usp.br}}}

\begin{document}

\setlength{\abovedisplayskip}{1.0ex}
\setlength{\abovedisplayshortskip}{0.8ex}

\setlength{\belowdisplayskip}{1.0ex}
\setlength{\belowdisplayshortskip}{0.8ex}
\maketitle

\abstract{We provide an equivalent characterization
  for the existence of one-dimensional irrational rotation factors of
  conservative torus homeomorphisms that are not eventually annular.
  It states that an area-preserving non-annular torus homeomorphism
  $f$ is semiconjugate to an irrational rotation $R_\alpha$ of the
  circle if and only if there exists a well-defined speed of rotation
  in some rational direction on the torus, and the deviations from the
  constant rotation in this direction are uniformly bounded. By means
  of a counterexample, we also demonstrate that a similar
  characterization does not hold for eventually annular torus
  homeomorphisms.\smallskip


  }

\noindent
\section{Introduction}

The question of irrational rotation factors, that is, the existence of
a semiconjugacy or conjugacy to an irrational rotation of the circle,
is a classical problem in dynamical systems theory. An equivalent
formulation is to ask for the existence of continuous eigenfunctions
of the associated Koopman operator with irrational eigenvalues.
Typically, this issue is addressed by using the powerful tools of
KAM-theory, which even yield smooth conjugacies (`smooth
linearization'). The price to pay for this, however, is the
requirement of strong assumptions concerning the smoothness of the
considered systems and the arithmetic properties of the involved
rotation numbers.  Moreover, this approach is mostly restricted to
systems which are small perturbations of the underlying rotations.

At the same time, the very first result relating non-linear dynamics
to irrational circle rotations is the celebrated Poincar\'e
Classification Theorem, which states that an orientation-preserving
homeomorphism of the circle is semiconjugate to an irrational rotation
if and only if its rotation number is irrational \cite{poincare:1885}.
It is remarkable that this statement draws strong conclusions from
purely topological assumptions, and no restrictions on the rotation
number other than its irrationality are needed. The fact that the
existence of full conjugacies cannot generally be expected in a
topological setting is well-known and demonstrated by classical
examples of Denjoy \cite{Denjoy1932Courbes}.

Even 130 years after Poincar\'e's contribution, similar results in
this direction are quite rare. In recent years, however, there has
been substantial progress on `topological linearization' in a number
of situations and system classes, including skew-products over
irrational rotations \cite{jaeger/stark:2006}, mathematical
quasicrystals \cite{Aliste2010QuasicrystalTranslations},
reparametrisations of irrational flows
\cite{AlisteJaeger2012AlmostPeriodicStructures}, area-preserving torus
homeomorphisms \cite{jaeger:2009b} and dynamics on circle-like
continua \cite{JaegerKoropecki2014DecomposableContinua}. Most
importantly, some general pattern and methods start to emerge. In
particular, a common element in most of the proofs is the
identification of a suitable dynamically defined partition of the
phase space carrying a circular order structure.

Our aim here is to push forward this line of research by providing a
more or less complete description of the situation concerning the
existence of one-dimensional irrational rotation factors of
area-preserving homeomorphisms of the two-torus. For the non-annular
case we provide an equivalent characterization in terms of rotational
behavior, whereas in the eventually annular case we show that an
analogous statement is not valid. This complements a previous result
in \cite{jaeger:2009b}, which treats the existence of two-dimensional
rotation factors. Here, we note that a periodic point free
area-preserving torus homeomorphism is annular if there exists an
essential (but not fully essential) invariant continuum or,
equivalently, an essential invariant open annulus. It is eventually
annular if it has an annular iterate. The general definitions are
given in Section~\ref{Preliminaries} below. In the eventually annular
case, the dynamics of the respective iterate of $f$ can be embedded in
a compact annulus. This situation is quite different from the
non-annular case,\foot{In the context of this note, we use the term
  {\em `non-annular'} is the sense of `not eventually annular'.} where
the dynamics truly involve the full topology of the torus. Annular
homeomorphisms have been extensively studied in own right (see, for
example,
\cite{franks/lecalvez:2003,beguin/crovisier/leroux:2004,jaeger:2010a}
and references therein) and significant information concerning their
dynamical behavior is available. However, as we exemplify in
Section~\ref{Counterexample}, the question of irrational rotation
factors is more intricate than in the non-annular case and does not
have a similar solution.

Given a homeomorphisms $f$ of the two-torus $\torus=\R^2/\Z^2$, we
denote its lift to $\R^2$ by $F$. We say $v=(p,q)\in\Z^2\smin\{0\}$ is
{\em reduced} if $\gcd(p,q)=1$ and call $w\in\Z^2$ {\em complementary}
if $\det(w,v)=1$. Suppose that for some reduced integer vector $v$ we
have $F(z+v)=F(z)+v$ for all $z\in\R^2$. Note that this property is
independent of the choice of the lift. Moreover, when $f$ is homotopic to the
identity, this holds for all reduced integer vectors. The {\em
  rotation interval of $F$ in the direction of $v$} is given by
\[
\rho_v(F)  \ = \ 
\bigcap_{n\in\N}\overline{\bigcup_{m\geq n} K(F,m)} \ ,
\]
where $K(F,m) = \{\langle F^m(z)-z,v\rangle/m \mid z\in\R^2\}$. It is always a
compact interval \cite{misiurewicz/ziemian:1989}. If $\rho_v(F)=\{\alpha\}$ for
some $\alpha\in\R$ and moreover there exists a constant $C>0$ such that
\[
\left|\langle F^n(z)-z,v\rangle - n\alpha\right| \ \leq C \
\]
for all $n\in\N$ and $z\in\R^2$, then we say $f$ has {\em bounded deviations} (from
the constant rotation) {\em in the direction of $v$}.  Denote the rotation by
$\alpha\in\R$ on $\kreis$ by $R_\alpha$. Then our main result reads as follows.
\begin{thm} \label{t.main} Suppose $f$ is a non-annular
  area-preserving homeomorphism of $\torus$.
  Then $f$ is semiconjugate to an irrational rotation $R_\alpha$ on
  $\kreis$ if and only if there exist a reduced integer vector $v$
  and a positive integer $k$ such that $\rho_v(F)=\{\alpha/k\|v\|^2\}$
  and $f$ has bounded deviations in the direction of $v$.
\end{thm}

Suppose that $f$ is homotopic to the identity and there exist a vector
$\rho=(\alpha,\beta)\in\R^2$ and a constant $C>0$ such that
\[
|F^n(z)-z-n\rho| \ \leq \ C 
\]
for all $n\in\N$ and $z\in\R^2$. In this case, we say $f$ is a {\em
  pseudo-rotation}
{\em with uniformly
bounded deviations}. If $\rho$ is totally irrational (i.e., if $\alpha, \beta$ and $\alpha/\beta$ are all irrational), then
Theorem~\ref{t.main} can be applied twice to obtain a semiconjugacy to
the irrational rotation by $\rho$ on the two-torus. This is precisely
the statement of Theorem C in \cite{jaeger:2009b}. However, we
actually use \cite[Theorem C]{jaeger:2009b} in the proof of
Theorem~\ref{t.main}, in order to treat exactly the above-mentioned
situation. The new ingredient we provide is a complemen\-tary argument
that addresses the cases where deviations are only bounded in one
direction, or where $f$ is homotopic to a Dehn twist. This part also
works under the weaker assumption of nonwandering dynamics and is
stated as Theorem~\ref{t.complementary_case}. The proof is based on
the concepts of strictly toral dynamics and dynamically essential
points developed in \cite{KoropeckiTal2012StrictlyToral}.

Section~\ref{Preliminaries} collects all the required preliminaries
and provides an elementary, but crucial lemma on strictly toral
dynamics. The proof of Theorem~\ref{t.main} is given in
Section~\ref{Homotopic_to_Id}, and in Section~\ref{Counterexample} we
provide a simple counterexample showing that the statement of
Theorem~\ref{t.main} is false in the eventually annular case.
\medskip

\noindent{\bf Acknowledgments.} This collaboration was carried out in
the framework of the Brazilian-European exchange program BREUDS. TJ
was supported by the German Research Council (Emmy-Noether grant OE
1721/2-1). FT was partially supported by CNPq grant 3004474/2011-8 and
Fapesp 2011/16265-8.

\section{Definitions and preliminaries} \label{Preliminaries}

\subsection{Circloids.}

Let $S$ be a two-dimensional manifold. A continuum $A\ssq S$ is called
{\em annular}, if it is the intersection of a nested sequence of
annuli $A_n$ such that each $A_{n+1}$ is essential in $A_n$ (not
contained in a topological disk $D\ssq A_n$.) An equivalent definition
is to require that $A$ has an annular neighborhood $\cA$ which it
separates into exactly two connected components, both of which are
again homeomorphic to the open annulus. We say $A$ is {\em essential}
if this annular neighborhood $\cA$ is essential in $S$.  We call an
annular continuum $C\ssq S$ with annular neighborhood $\cA$ a
{\em circloid}, if it does not contain a strictly smaller annular continuum
which is also essential in $\cA$. Note that $C$ may contain a
non-essential annular continuum as a subset, for example when it has
non-empty interior.

A compact set $A\ssq \A=\R\times\kreis$ is essential if
$\A\smin A$ has two unbounded connected components. In this case, we
denote the component which is unbounded to the right by $\cU^+(A)$ and
the one unbounded to the left by $\cU^-(A)$. Note that if $A\ssq\A$ is
an essential annular continuum, then $\A\smin C=\cU^-(C) \cup
\cU^+(C)$.  Further, we let $\cU^{+-}(A) = \cU^-(\partial\cU^+(A))$
and $\cU^{-+}(A)=\cU^+(\partial\cU^-(A))$, and define the same notions
for longer alternating sequences of the symbols $-$ and $+$ in the
analogous way.  This yields a simple procedure to obtain essential
circloids from arbitrary essential compact sets.
\begin{lem}[{\cite[Lemma
    3.2]{jaeger:2009b}}]\label{l.bounding_circloids}
  Suppose $A\ssq \A$ is compact and essential. Then
\[
\cC^+(A) \ = \ \A\smin (\cU^{+-}(A) \cup \cU^{+-+}(A)) \eqand \cC^-(A) \ = \
\A\smin (\cU^{-+}(A) \cup \cU^{-+-}(A))
\]
are both essential circloids.
\end{lem}

Given two essential circloids $C,C'\ssq\A$, we write $C\prec C'$ if
$C'\ssq \cU^+(C)$ and $C\preccurlyeq C'$ if $C'\ssq \cU^+(C)\cup C$.
Further, we write $[C,C'] = \A \smin (\cU^-(C)\cup \cU^+(C'))$ and
$(C,C')=\cU^+(C) \cap \cU^-(C')$. We will need the following
elementary observation.
\begin{lem} \label{l.space_between_circloids} If $C\preccurlyeq C'$ and $C\neq
  C'$, then $(C,C')$ is non-empty.
\end{lem}
\proof Suppose for a contradiction that $(C,C')$ is empty. Then, since
$\cU^-(C')\cap \cU^+(C)=\emptyset$, the set $C^*=\A\smin(\cU^-(C')\cup
\cU^+(C))$ is an essential annular continuum. However, $C\preccurlyeq C'$
implies that $\cU^-(C)\ssq \cU^-(C')$ and $\cU^+(C')\ssq \cU^+(C)$. Therefore
$C^*$ is contained both in $C$ and $C'$, and by minimality of circloids amongst
annular continua, with respect to inclusion, we obtain $C^*=C=C'$.  \qed\medskip

\subsection{Strictly toral dynamics.}

An open set $U\subset S$ is called {\it inessential} if every loop
whose image lies in $U$ is homotopically trivial in $S$, otherwise $U$
is {\it essential}. A general set $E\subset S$ is called inessential
if it has an inessential neighborhood, otherwise $E$ is said to be
essential. Note that for subcontinua of $\A$, this coincides with the
terminology used above. We identify $\R^2$ with the universal covering space of $\T^2$ and let $\pi:\R^2\to\T^2$ be the covering projection.  A set $E\ssq \T^2$ whose complement is
inessential is called {\it fully essential}.  An open set $U\ssq \T^2$
is called {\it annular} if it is neither inessential, nor fully
essential. If $U$ is connected, this is equivalent to saying that the union of $U$ with the
inessential components of its complement is a topological open
annulus.  If $U$ is open and $i:\Pi_1(U)\to \Pi_1(\T^2)$ is the
natural inclusion of its first homotopy group into the first
homotopy group of $\T^2$, then $U$ is inessential if and only if the
image of $i$ is trivial, $U$ is fully essential if and only if $i$ is
onto, and $U$ is annular if and only if the image of $i$ is
homomorphic to $\Z$.  Note that, if $U$ is a connected fully essential
set, then $\pi^{-1}(U)$ is connected, and that any two open fully
essential subsets of $\T^2$ must intersect.

Given a homeomorphism $f:\T^2\to\T^2$, we say that $x\in\T^2$ is an
{\it inessential point} for $f$ or that $x$ is a dynamically
inessential point if there exists a neighborhood $U$ of $x$ such that
$\bigcup_{i\in\Z}f^{i}(U)$ is inessential, otherwise $x$ is called a
dynamically essential point. The set of dynamically inessential points
is denoted by Ine($f$) and is open. Its complement is denoted by
Ess($f$). If $f$ is nonwandering, then $x$ is an inessential point if
and only if it is contained in some periodic open topological disk. By
Brouwer's theory, every orientation preserving map of the plane with nonwandering points has a fixed point. As an open topological disk is homeomorphic to the plane,  it follows that
\begin{lem}
If $f$ is nonwandering and Ine($f$) is not empty, then $f$ has a periodic point.
\end{lem} 

We say a homeomorphism $f$ of the two-torus is {\it annular} if there
exists $M>0$, an integer vector $v \in \Z^2$ and a lift $F$ of $f$
such that, for any $x \in \R^2$ and any $n\in\N,\, \vert \langle
F^n(x)-x, v\rangle \vert <M.$ It is {\em eventually annular} if it has
an annular iterate. Recall that we say $f$ is {\em non-annular} if it
is not eventually annular. As stated above, if $f$ is area-preserving
and has no periodic points, then $f$ is annular (eventually annular)
if and only if $f$ has an invariant (periodic) essential open annulus
\cite[Proposition 3.9]{jaeger:2009c}.  A homeomorphism $f:\T^2\to\T^2$
is said to be {\it strictly toral} or to have strictly toral dynamics
if, for all $n\in \N$, $f^n$ is not annular and Fix($f^n$) is an
inessential set. The following is immediate.
\begin{lem}
  If $f$ is nonwandering and $f^n$ non-annular for all $n\in\N$, then every
  essential periodic open set is fully essential.
\end{lem}  

An important piece in the understanding the behavior of essential
points is to understand the boundedness properties of the dynamics.
Let $f$ be a torus homeomorphism homotopic to the identity and $F$ a
lift of $f$. We say that $\rho$ is a {\em rotation vector} for $F$ if
there exists some $z$ in $\R^2$ such that
$\lim_{n\to\infty}\frac{F^n(z) -z}{n}= \rho$. If $F$ has a single
rotation vector $\rho$, we say that $f$ is a {\em pseudo-rotation}, in
which case for all $z$ in $\R^2$, $\lim_{n\to\infty}\frac{F^n(z)
  -z}{n}= \rho$.  

Call $D\subset\R^2$ a {\it fundamental domain} if the restriction of $\pi$ to $D$ is one-to-one and $\pi(D)=\T^2$. We have

\begin{lem}\label{lm:diameterunbounded}
  Suppose $f$ is a torus homeomorphism for which no power of $f$ is a
  pseudo-rotation with uniformly bounded deviations. Then for any
  fundamental domain $D\subset \R^2$, any lift $F$ of $f$, and any
  $K>0$, there exists $n>0$ such that $F^n(D)$ has diameter larger
  than $K$.
\end{lem}
\begin{proof}
  Fix a lift $F$ of $f$, and first consider the case where $f$ is not
  homotopic to a periodic homeomorphism. Then there exists some simple
  closed curve $\alpha:[0,1]\to \T^2$ such that, if $\widetilde \alpha$ is
  a lift of $\alpha$ to $\R^2$, then $\lim_{n\to\infty}\Vert
  F^n(\widetilde \alpha(0))-F^n(\widetilde \alpha(1))\Vert= \infty$, and so
  for any fundamental domain $D$ the diameter of $F^n(D)$ is not
  uniformly bounded.

 Now consider the case where $f$ is homotopic to a periodic  homeomorphism, in which case there exists some $\overline n$ such that $g =
  f^{\overline n}$ is homotopic to the identity, and let $G=F^{\overline n}$. Let $\rho$ be a rotation 
vector of $F$, and assume that there exists some fundamental domain $D\subset \R^2$ and some $K>0$
such that $G^n(D)$ has diameter smaller than or equal to $K$ for all $n>0$. We will show that, for all $z\in\R^2$ and all $n\in\N$, $\Vert G^n(z)-z-n\rho\Vert\le 2K$, which implies the lemma.

Assume, for a contradiction that this fails for some $z_0$ and $n_0$. We can assume that $z_0\in D$. Let $w= \frac{G^{n_0}(z)-z-{n_0}\rho}{\Vert G^{n_0}(z)-z-{n_0}\rho\Vert}$ and let $D_w(n, z) = \langle G^n( z)- z  - n\rho, w\rangle$. If $z\in D$, then by the contradiction hypothesis $\Vert G^{n_0}(z)-G^{n_0}(z_0)\Vert<K$ and $\Vert z- z_o\Vert<K$, and therefore $\vert D_w(n_0, z)-D_w(n_0, z_0)\vert< 2K$ which implies that $D_w(n_0, z)>0$ for all $z\in D$ and also for all $z\in \R^2$ as $D$ is a fundamental domain. Since the function $D_w(n_0,z)$ lifts a continuous function on the torus, there exists $a>0$ the minimum value of $D_w(n_0, z)$. As $D_w(n_1+ n_2,z)= D_w(n_1, z)+ D_w(n_2, G^{n_1}(z))$, we have that for all $k$ and all $z$, $D_w(kn_0, z)\ge ka$ and so 
$\lim_{n\to\infty}\Vert\frac{G^n(z) -z-n\rho}{n}\Vert$ cannot be null, which contradicts $\rho$ being a rotation vector for $G$.
\end{proof}

The following proposition is the main new result of this subsection.
Its proof is inspired by those of Theorem 8 and Proposition 9 of
\cite{GuelmanKoropeckiTal2014transitivity}.


\begin{prop} \label{p.essential_annulus} Suppose $f$ is a strictly
  toral nonwandering torus homeomorphism and that no power of $f$ is
  a pseudo-rotation with bounded deviations. Then for any
  neighbourhood $U$ of an essential point $x$ there exists $n\in\N$
  such that $U\cup f^n(U)$ is essential.
\end{prop}

\begin{proof}
  By maybe taking a subneighborhood, we can assume that $U$ is
  inessential and contained in $B_{1/2}(x)$. Let us fix a connected
  component $\widetilde U$ of $\pi^{-1}(U)$.  Since $f$ is strictly
  toral, every invariant or periodic essential set is fully essential
  and invariant, and since $x$ is an essential point, this implies
  that $O = \bigcup_{i\in\Z} f^i(U)$ is fully essential, and $O$ is trivially invariant.
  Therefore there exist simple closed curves $\alpha_1, \alpha_2:
  \kreis \to \T^2$ whose images lie in $O$ and which generate the
  fundamental group of $\T^2$. This implies that the connected
  components of the complement of $\pi^{-1}(\alpha_1\cup\alpha_2)$ are
  uniformly bounded.  Note also that for each point $z$ in
  $\pi^{-1}(\alpha_1\cup\alpha_2)$, there exists some integer $i_{z}$
  and some $v_{z}\in\Z^2$ such that $z\in F^{i_{z}}(\widetilde U)+
  v_{z}$.

  Let $D\subset \R^2$ be a fundamental domain, and consider the
  connected set $D'$ which is the union of $D$ with all connected
  components of the complement of $\pi^{-1}(\alpha_1\cup\alpha_2)$
  intersecting $D$. Then $D'$ is a bounded set, and its boundary is
  contained in $\pi^{-1}(\alpha_1\cup\alpha_2)$. Let $R$ be larger
  than the diameter of $D'$. By compactness, one can find integers
  $i_j, 1\le j \le k$, and integer vectors $v_j, 1\le j \le k$ such
  that $\partial D'\subset \bigcup_{j=1}^{k} F^{i_j}(\widetilde
  U)+v_j$, and $\bigcup_{j=1}^{k} F^{i_j}(\widetilde U)+v_j$ is
  connected.  Since no power of $f$ is a pseudo-rotation with
  uniformly bounded deviations, by Lemma \ref{lm:diameterunbounded}
  there exists some $n>0$ such that the diameter of $F^n(D)$ is larger
  than $k(k+1)R$. This implies that the diameter of
  $\bigcup_{j=1}^{k} F^{i_j+n}(\widetilde U)+v_j$ is greater than
  $k(k+1)R$ and thus, for some $j_0\le k$, the diameter of $
  F^{i_{j_0}+n}(\widetilde U)+v_{j_0}$ is larger than $(k+1)R$.

  But this implies that $ F^{i_{j_0}+n}(\widetilde U)+v_{j_0}$
  intersect at least $k+2$ integer translates of $D'$, and so $
  F^{i_{j_0}+n}(\widetilde U)+v_{j_0}$ must intersect at least $k+1$
  integer translates of $\partial D'$.  Let $w_l,\, 1\le l \le k+1$ be
  vectors in $\Z^2$ such that $ F^{i_{j_0}+n}(\widetilde U)+v_{j_0}$
  intersects $\partial D'+w_l$ for each $l$. As each copy $\partial
  D'+w_l$ is covered by the $k$ sets $F^{i_j}(\widetilde U)+v_j+w_l$,
  there exists $w_{l_1}\not=w_{l_2}$ and some $j_1 \le k$ such that
  $F^{i_{j_0}+n}(\widetilde U)+v_{j_0}$ intersects both
  $F^{i_{j_1}}(\widetilde U)+v_{j_1}+ w_{l_1}$ and $
  F^{i_{j_1}}(\widetilde U)+v_{j_1}+ w_{l_2}$.

  This implies that $F^{i_{j_0}+n- i_{j_1}}(\widetilde
  U)+v_{j_0}-v_{j_1}$ intersects both $\widetilde U +w_{l_1}$ and
  $\widetilde U+w_{l_2}$. Therefore, the connected components of
  $\pi^{-1}\left(f^{i_{j_0}+n- i_{j_1}}(U) \cup U\right)$ do not
  project injectively on $\T^2$ and so $f^{i_{j_0}+n- i_{j_1}}(U) \cup
  U$ is an essential set.
\end{proof}

\section{Proof of Theorem~\ref{t.main}}
\label{Homotopic_to_Id}

We start with the ``only if''-direction. We identify the fundamental
group $\Pi_1(\T^d)$ of the $d$-dimensional torus with $\Z^d$ and
denote the action of a continuous map $\psi:\T^d\to\T^k$ on the
fundamental groups by $\psi_*:\Pi_1(\T^d)\to\Pi_1(\T^k)$. In the
following, if $h:\torus\to\kreis$ semiconjugates a torus homeomorphism
$f$ to $R_\alpha$ and $H$ and $F$ are lifts, then we implicitly assume
these to be chosen such that $H$ semiconjugates $F$ to $x\mapsto
x+\alpha$ on $\R$. Given $v=(a,b)$ in $\R^2$, denote by $v^{\perp}=(b,-a)$.
\begin{lem}
  Suppose a homeomorphism $f$ of the two-torus is semiconjugate to an
  irrational rotation $R_\alpha$ on $\kreis$ via the factor map
  $h:\torus\to\kreis$. Then $h_*:\Pi_1(\T^2)\to\Pi_1(\kreis)$ is given
  by $u\to k\langle u,v\rangle$ for some reduced integer vector $v$
  and a positive integer $k$. Moreover, we have
  $f_*(v^\perp)\in\{v^\perp,-v^\perp\}$,
  $\rho_{v}(F)=\{\alpha/k\|v\|^2\}$ and $f$ has bounded deviations in
  the direction of $v$. In addition, if $w$ is complementary to
  $v^\perp$, then $f_*(w)=w+mv^\perp$ for some $m\in\Z$.
\end{lem}
\proof First, suppose for a contradiction that $h_*=0$ and let
$H:\R^2\to \R$ be a lift of $h$. Then $\sup_{z\in\R^2}|H(z)|
<\infty$. However, since $H(F^n(z)) = H(z)+n\alpha$, this contradicts
the fact that $\alpha\neq 0$.  Hence, the kernel of $h_*$ is
homomorphic to $\Z$ and generated by some reduced integer vector
$\tilde v$, so if we let $v=\tilde v^\perp$ we have $h_*(u)=k\langle
u,v\rangle$ for some $k\in\Z\smin\{0\}$. By replacing $v$ with $-v$ if
necessary, we may assume $k>0$. Let $w$ be a reduced integer vector
complementary to $v^\perp$ and $A=(w,v^\perp)\in\textrm{SL}(2,\Z)$.
Then $A$ induces a linear torus homeomorphism $\phi_A$ with lift
$\Phi_A:z\mapsto Az$, and $\tilde h=h\circ\phi_A$ semiconjugates
$\tilde f=\phi_A^{-1}\circ f\circ\phi_A$ to $R_\alpha$.  Since $\tilde
h_*=h_*\circ A$ and $A$ sends $(1,0)$ to $w$ and $(0,1)$ to $v^\perp$,
$\tilde h_*$ sends $(1,0)$ to $k$ and $(0,1)$ to zero. This implies
that $\tilde h$ is homotopic to the map $z\mapsto k\pi_1(z) \bmod 1$.
By replacing the circle by a $k$-fold covering and $R_\alpha$ by
$R_{\alpha/k}$, we may assume that $\tilde h$ is homotopic to $\pi_1$
and semiconjugates $\tilde f$ to $R_{\alpha/k}$.

Let $\tilde F$ and $\tilde H$ be lifts of $\tilde f,\tilde h$, again
chosen such that $\tilde H$ semiconjugates $\tilde F$ to the
translation by $\alpha/k$ on $\R$. The fact that $\tilde h$ is
homotopic to $\pi_1$ implies that $C=\sup_{z\in \R^2} |\tilde
H(z)-\pi_1(z)|<\infty$. This further yields
\begin{eqnarray}
  \lefteqn{|\pi_1(\tilde F^n(z)-z)-n\alpha/k| \  = }\nonumber\\
  & = &  |\pi_1(\tilde F^n(z))-\tilde H(\tilde F^n(z))+
   \label{e.bounded_deviations_from_sc}
  \underbrace{\tilde H(\tilde F^n(z))-\tilde H(z)-n\alpha/k}_{=0}
   + \tilde H(z)-\pi_1(z)| \\
  & \leq &  |\pi_1(\tilde F^n(z))-\tilde H(\tilde F^n(z))| 
  + |\tilde H(z)-\pi_1(z)| \ \leq \ 2C \ . \nonumber
\end{eqnarray}
Applied to $n=1$, this means in particular that $|\pi_1(\tilde
F(0,0)-\tilde F(0,l))| = l|\pi_1(\tilde f_*(0,1))|\leq 4C$ for all
$l\in\N$, so that $\pi_1(f_*(0,1))=0$.  Hence, $(0,1)$ is an
eigenvector for the action of $\tilde f_*$. Since $\tilde f$ is a
homeomorphisms and thus $\tilde f_*$ has to permute reduced integer
vectors, it must send $(0,1)$ either to itself or to $(0,-1)$. Thus
$f_*(v^\perp)\in \{v^\perp,-v^\perp)$. Similarly, we have
$|\pi_1(\tilde F(l,0)-l-\tilde F(0,0)| \leq 4C$ for all $l\in\N$,
which implies that $\pi_1(\tilde f_*(1,0))=1$, so $\tilde
f_*(1,0)=(1,m)$ for some $m\in\Z$. Therefore $f_*(w)=w+mv^\perp$.

Finally, (\ref{e.bounded_deviations_from_sc}) directly implies that
$\tilde f$ has rotation number $\alpha/k$ and bounded deviations in
the direction of $(1,0)$. Going back to the original coordinates, we
obtain that $F=\Phi_A\circ F\circ \Phi_A^{-1}$ has a well-defined
speed of rotation and bounded deviations orthogonal to $A(0,1)^t=u^t$,
and the rotation number in this direction is
$\alpha/k\cdot\frac{\langle w,u^\perp\rangle}{\|u\|^2} =
\alpha/k\|u\|^2$ (see also \cite[Proof of
Proposition~2.1]{jaeger:2009b}). This proves the statement.
\qed\medskip

By choosing $v^\perp$ and a complementary vector $w$ as a basis of
$\Z^2$, we can thus determine the possible homotopy types of $f$.
\begin{cor}\label{c.homotopy_types}
  If a torus homeomorphisms $f$ is semiconjugate to an irrational
  rotation of the circle, then the action $f_*:\Z^2\to\Z^2$ is
  conjugate to a linear transformation given by a matrix of the form
  $\twomatrix{\pm 1}{m}{0}{1}$ with $m\in\Z$.
\end{cor}
\medskip

In order to prove the converse direction, we may assume by means of a
linear coordinate change that $v=(1,0)$.  If we let
$\A=\R\times\kreis$ and denote by $\wh F:\A\to\A$ the respective lift
of $f$ and by $\pi_1:\A\to\R$ the projection onto the first coordinate, then due to the bounded deviations in the $v$-direction we
have
\begin{equation}
  \label{eq:1}
  \left|\pi_1\left(\wh F^n(z)-z\right)-n\alpha\right| \ \leq \ C
\end{equation}
for some constant $C\geq 0$.  Suppose that some iterate $f^p$ of $f$
has a well-defined rotation number and bounded deviations also
orthogonal to $v$.  Since we exclude the eventually annular case, this
means that $\rho_{v^\perp}(F^p) =\{\beta\}$ for some $\beta\in\R$ which
is not rationally related to $\alpha$ \cite[Proposition
3.9]{jaeger:2009b}.  It is easy to see that this can neither happen if
$f$ is orientation-reversing nor if it is homotopic to a Dehn-twist.
Thus, in the light of Corollary~\ref{c.homotopy_types}, $f$ has to be
homotopic to the identity.  In this case, the existence of a
semiconjugacy to the two-dimensional rotation by $\rho=(\alpha,\beta)$
on $\torus$ follows directly from \cite[Theorem C]{jaeger:2009b}.

Hence, it remains to treat the case no iterate of $f$ is a
pseudo-rotation with bounded deviations. As we have mentioned, the
argument which we will apply here only requires $f$ to be
nonwandering. Note that if $h$ semiconjugates $f$ to $R_{\alpha/k}$
and $\tau(x)=kx\bmod 1$, then $\tau\circ h$ semiconjugates $f$ to
$R_\alpha$. Hence, in order to complete the proof of
Theorem~\ref{t.main}, it suffices to show the following statement.

\begin{thm} \label{t.complementary_case} Suppose $f$ is a nonwandering
  non-annular torus homeomorphism and $\rho(F)\ssq \{\alpha\}\times\R$
  for some $\alpha\in\R\smin\Q$.  Moreover, assume that $f$ has
  bounded deviations in the direction of $(1,0)$, but no iterate of
  $f$ is a pseudo-rotation with uniformly bounded deviations. Then $f$
  is semiconjugate to the irrational rotation $R_\alpha$ on $\kreis$.
\end{thm}
\proof Due to (\ref{eq:1}), the sets $A_r=\bigcup_{n\in\Z} \wh
F^n(\{r-n\alpha\}\times \kreis)$ are compact, and it is easy to check
that they satisfy $\wh F(A_r)=A_{r+\alpha}$ and $A_{r+1} = T(A_r)$, where
$T:\A\to\A,\ (x,y)\mapsto(x+1,y)$. Moreover, these relations carry
over to the circloids $C_r=\cC^+(A_r)$. In addition, the monotonicity
of the construction implies $C_r\preccurlyeq C_s$ whenever $r\leq s$.
We now define 
\[
H : \A \to \R \quad , \quad   z\mapsto \sup\{r\in\R \mid z\in \cU^+(C_r)\} \ .
\] 
Then $H\circ F(z) = \sup\{r\in\R \mid F(z)\in \cU^+(C_r)\} \ = \
\sup\{r\in\R \mid z\in \cU^+(C_{r-\alpha})\} \ = \ H(z)+\alpha$. In
the same way, one can see that $H\circ T = T\circ H$. Hence, $H$
projects to a map $h:\torus\to\kreis$ that satisfies $h\circ
f=R_\alpha\circ h$. If $h$ is continuous, then it follows immediately
from the minimality of $R_\alpha$ that $h$ is also onto. Thus, it only
remains to check the continuity of $h$.

In order to do so, however, it suffices to prove that the circloids
$C_r\ssq \A$ with $r\in\R$ are pairwise disjoint. This fact is shown
on \cite[page 615]{jaeger:2009b} ({\em Construction of the
  semi-conjugacy}), and we refrain from repeating the argument here.
Thus, in order to complete the proof, it remains to show the
disjointness of the circloids.

To that end, let $r'<s'$ and suppose without loss of generality that
$s'-r'\leq1$. Let $r=(2r'+s')/3$ and $s=(r'+2s')/3$. If $C_r=C_s$, then $C_t=C_r$
for all $t\in(r,s)$, and we may choose $t$ of the form $t=r+p\alpha-q$ for
suitable integers $p,q$. Then $C_r=C_{r+p\alpha-q}=T^{-q}\circ \wh F^p(C_r)$,
which implies that $\rho_v(F)=\{p/q\}$, contradicting the irrationality of
$\alpha$.

Hence, we have $C_r\neq C_s$, and due to
Lemma~\ref{l.space_between_circloids} we can find an open disk $\wh
U\ssq (C_r,C_s)$. Let $U=\pi(\wh U)$.  According to
Proposition~\ref{p.essential_annulus}, there exists an integer
$n\in\N$ such that $f^n(U)\cup U$ is essential. This implies in
particular that $f^n(U)\cap U\neq \emptyset$, so that we can choose
an integer $m$ such that $\wh V=T^{-m}\circ \wh F^n(\wh U)$ intersects $\wh
U$.

Since $\wh U\ssq (C_r,C_s)$ and $\wh V\ssq
(C_{r+n\alpha-m},C_{s+n\alpha-m})$, this is only possible if
$|n\alpha-m|<(s'-r')/3$. This yields that $\wh U\cup \wh V\ssq
(C_{r'},C_{s'})$, and as a consequence $\wh V$ cannot intersect any
translate of $\wh U$. Since $U\cup f^n(U)$ is essential, this means
$\wh U\cup \wh V$ is essential as well. We thus obtain an essential
open set between $C_{r'}$ and $C_{s'}$ which separates the two
circloids, so that these must be disjoint. Since $r'<s'$ were
arbitrary, this completes the proof.  \qed\medskip

\section{A counterexample in the annular case}\label{Counterexample}

In this section we sketch the construction of an area-preserving
transitive annular irrational pseudo-rotation of the $2$-torus with
uniformly bounded deviations that is not semi-conjugated to any
irrational rotation.

To begin, Besicovitch in \cite{Besicovitch1951Example}( see also
\cite{Schnirelman30Example}) has shown the existence of a transitive
homeomorphism $g:\kreis\times \R$, which is a skew-product over an
irrational rotation, that is, $g$ is of the form $g(x,y)= (x+\alpha,
y+\phi(x))$ with $\alpha$ irrational.  We claim that the only possible
semi-conjugacies between $g$ and an irrational rotation are the
trivial ones, given by the projection onto the first coordinate
composed with a uniform rotation. Indeed, if $(x_0, y_0)$ is a point
with dense forward orbit, then for any $(x, y)$ there exists a
sequence $(n_k)_{k\in\N}$ such that $g^{n_k}(x_0, y_0)$ converges to
$(x,y)$ and, in particular, $n_k\alpha$ converges to $x - x_0$.  If
$h:\kreis\times \R \to \kreis$ is the map semiconjugating $g$ to a
rotation $R_{\alpha}$, then $h(x_0, y_0)- (x-x_0)= \lim_{k\to\infty}
R_{\alpha}^{n_k} h(x_0) = \lim_{k\to\infty}h (g^{n_k}(x_0, y_0))=
h(x,y)$ and so $h(x,y)-h(x_0, y_0) = x-x_0$.

Now, let $\psi:\kreis\times\R\to \kreis\times (-1,1)$ be defined as
$$\psi(x, y)= (x + \sin( \log(\vert y \vert +1)), \frac{y}{\vert y \vert+1}).$$
One can easily verify that, if $f$ is the homeomorphism of
$\kreis\times (-1,1)$ given by $f = \psi\circ g\circ \psi^{-1}$, then
$f$ can be extended to a homeomorphism of the closed annulus
$\A=\kreis\times [-1,1]$ by defining $f(x, 1)= (x+\alpha, 1)$ and
$f(x, -1)= (x+\alpha, -1)$. It is also immediate that $f$ has
uniformly bounded deviations from the rigid rotation by $\alpha$, and
that the $\psi$ image of any fiber $\{x\}\times \R$ accumulates on all
boundary points of $\A$.

Now, if by contradiction there exists $h$ that semi-conjugates $f$ to
the rotation $R_{\alpha}$, then $h\circ\psi$ semi-conjugates $g$ to
$R_{\alpha}$. This implies that, for each $\theta \in \kreis$,
$h^{-1}(\theta)$ must contain $\psi(x, \R)$ for some $x$, and by
continuity $h^{-1}(\theta)$ must contain both boundaries of $\A$,
which is impossible since the image of a fiber of the semiconjugation
must be disjoint from its iterates.


\end{document}